\documentclass[a4paper, 11pt]{amsart}

\usepackage[utf8]{inputenc}
\usepackage[all,cmtip]{xy}
\usepackage[english]{babel}
\usepackage{mathtools}
\usepackage{graphicx}
\usepackage{multicol}

\usepackage{amsmath,amssymb,amsthm,mathrsfs}
\usepackage[alphabetic]{amsrefs}
\let\amsrtimes=\rtimes
\usepackage{xfrac}
\usepackage{pifont}

\usepackage{array,tabularx}
\usepackage{xcolor}

\usepackage{tikz,tikz-cd}
\usetikzlibrary{positioning, shapes.geometric}
\usetikzlibrary{arrows,arrows.meta}
\usetikzlibrary{calc}
\tikzcdset{arrow style=tikz, diagrams={>={Computer Modern Rightarrow[width=5.5pt,length=3pt]}}}

\usepackage{enumitem}
\usepackage{cleveref}

\usepackage[retainorgcmds]{IEEEtrantools}

\theoremstyle{plain}
\newtheorem{thm}{Theorem}[section] 
\newtheorem*{thm*}{Theorem}
\newtheorem*{mainthm}{Main Theorem}
\newtheorem{prop}[thm]{Proposition}
\newtheorem{cor}[thm]{Corollary}
\newtheorem{lem}[thm]{Lemma}

\newtheorem*{fl}{Fico's Lemmata}

\counterwithin{equation}{section}

\theoremstyle{definition}
\newtheorem{defn}[thm]{Definition}

\newtheorem*{exa*}{Example}
\newtheorem{rem}[thm]{Remark}

\newtheorem*{rem*}{Remark}

\newtheorem*{warn}{Warning}

\makeatletter
\newcommand{\newreptheorem}[2]{\newtheorem*{rep@#1}{\rep@title}\newenvironment{rep#1}[1]{\def\rep@title{#2 \ref*{##1}}\begin{rep@#1}}{\end{rep@#1}}}
\makeatother

\theoremstyle{plain}
\newreptheorem{thm}{Theorem}
\newreptheorem{lem}{Lemma}
\newreptheorem{prop}{Proposition}
\newreptheorem{cor}{Corollary}

\usepackage{xparse}
\makeatletter
\RenewDocumentCommand{\title}{om}{%
  \IfNoValueTF{#1}
     {\gdef\shorttitle{On Certain Gyrations}}
     {\gdef\shorttitle{#1}}%
  \gdef\@title{#2}%
}
\makeatother

\usepackage{geometry}
\usepackage{libertine}
\usepackage[libertine]{newtxmath}
\renewcommand{\mathbb}{\varmathbb}

\input xypic 
\input xy 
\xyoption{all} 
\usepackage{amssymb} 
\usepackage{bbm}

\newcommand{\conn}{\ensuremath{\#}}

\newcounter{bean}

\title{On Fico's Lemmata and the Homotopy Type of Certain Gyrations}

\author{Sebastian Chenery}
\address{University of Bristol, School of Mathematics, Fry Building, Woodland Road, Bristol, BS8 1UG}
\email{seb.chenery@bristol.ac.uk}

\subjclass[2020]{Primary 57N65; Secondary 55P15, 57P10}
\keywords{Poincar\'e Duality complexes, gyrations, connected sums, sphere products.}

\begin{document}

\begin{abstract}
We undertake to determine the homotopy type of gyrations of sphere products and of connected sums, thereby generalising results known in earlier literature as ``Fico's Lemmata'' which underpin gyrations in their original formulation from geometric topology. We provide applications arising from recasting these results into the modern homotopy theoretic setting.
\end{abstract}

\maketitle

\section{Introduction}
\label{sec:intro}

A gyration is a certain surgery on the Cartesian product of a given manifold and a sphere. Originally defined by Gonz{\'a}lez Acu{\~n}a \cite{gonzalezacuna}, they have since appeared in several important contexts throughout geometry, including classifying circle actions on smooth manifolds \cites{duan, galaz-garcia--reiser} and in the topology of polyhedral products, where they feature in a foundational conjecture of Gitler--L{\'o}pez-de-Medrano \cite{gitler-ldm} which has been recently solved \cite{cfw}. Moreover, there has been much recent study of their homotopy theoretic properties by Huang \cite{huang_inertness24}, Huang-Theriault \cites{huangtheriault, HuangTher_stabilization}, Stanton-Theriault \cite{StanTher_skeleton-coH} and Basu-Ghosh \cite{basu-ghosh}. Also among these is \cite{ChenTher:gy_stab}, in which Theriault and this author developed a systematic approach to studying the homotopy class of the attaching map of the top-cell for gyrations of projective planes. Though these techniques were highly successful in answering our main question (that of \textit{gyration stability}, details in Section \ref{subsec:gystab}), we did not focus on providing clarity regarding the homotopy type of the gyrations involved. This is of course a hard problem in general, but the tools developed suggested future scope for obtaining the homotopy type of some gyrations in certain special circumstances. This is the thrust of this article, and as such it may be thought of as a spiritual successor to the methods of this previous work. 

More concretely, the full homotopy theoretic approach to defining gyrations has recently been formalised in \cite{HuangTher_stabilization} and goes as follows. Let $M$ be a path-connected \(n\)-dimensional Poincar\'{e} Duality complex with a single \(n\)-dimensional cell (examples include smooth, closed, oriented, simply-connected \(n\)-manifolds). We let \(\overline{M}\) be its \((n-1)\)-skeleton, and there is a homotopy cofibration 
\[
    S^{n-1} \xrightarrow{f_M} \overline{M} \rightarrow M
\] 
where \(f_M\) is the attaching map for the top-cell. Let \(k\geq 2\) be an integer and take a class \(\tau\in\pi_{k-1}(\mathrm{SO}(n))\). Using the standard linear action of \(\mathrm{SO}(n)\) on \(S^{n-1}\), define the map 
\[
    t:S^{n-1}\times S^{k-1}\rightarrow S^{n-1}\times S^{k-1}
\] 
by \(t(a, x)=(\tau(x)\cdot a,x)\). The \textit{\(k\)-gyration of \(M\) by \(\tau\)} is given by the pushout
\begin{equation*}
    \begin{tikzcd}[row sep=3em, column sep=3em]
        S^{n-1}\times S^{k-1} \arrow[r, "1\times \iota"] \arrow[d, "(f_M\times 1)\circ t"] & S^{n-1}\times D^k \arrow[d] \\
        \overline{M}\times S^{k-1} \arrow[r] & \mathcal{G}^k_\tau(M) 
    \end{tikzcd}
\end{equation*} 
where \(\iota\) is the inclusion of the boundary of the disc. When \(\tau\) is trivial the above pushout is equivalent to a \((k-1,n)\)-type surgery on \(M\times S^{k-1}\) with respect to this trivial choice. Otherwise the surgery is twisted by the action of \(\tau\), and for this reason the homotopy class \(\tau\) is referred to as a \textit{twisting} in the context of gyrations. The more geometric original formulation of Gonz{\'a}lez Acu{\~n}a is written in our notation as
\[
    \mathcal{G}_\tau^k(M)=\left((M-Int(D^n))\times S^{k-1}\right)\cup_t\left(S^{n-1}\times D^k\right)
\]
where \(D^n\subset M\) is an embedded \(n\)-disc centred at a chosen base point of \(M\). 

The guiding light of this paper is to generalise the result known colloquially as Fico's Lemmata - an affectionate term coined by Gitler--L{\'o}pez-de-Medrano \cite{gitler-ldm}*{Lemma 2}\footnote{This of course is after the first name Gonz{\'a}lez Acu{\~n}a himself, ``Fico'' being an abbreviation of Franciso.}. Despite the plural, the two parts of Fico's Lemmata are often combined into a single statement, and have since been expanded upon and reformulated, notably by Duan \cite{duan}; in our notation the result reads as follows.

\begin{fl} Let \(n\geq 5\) be an integer and \(\tau\in\pi_1(\mathrm{SO}(n))\).
    \begin{enumerate} 
        \item If \(M\) and \(N\) are two path-connected oriented \(n\)-manifolds, then there is a diffeomorphism
        \[
            \mathcal{G}_\tau^2(M\#N)\cong\mathcal{G}_\tau^2(M)\#\mathcal{G}_\tau^2(N).
        \]
        \item If \(q\geq p \geq 2\) are integers with \(q\geq 3\), then there is a diffeomorphism
        \[
        \pushQED{\qed}
            \mathcal{G}_\tau^2(S^p\times S^q)\cong (S^p\times S^{q+1})\#(S^{p+1}\times S^q).
        \qedhere
        \popQED
        \]
    \end{enumerate}
\end{fl}

The original proof of the above is purely geometric; our objective is to prove these, up to homotopy, for general \(k\)-gyrations. The method relies on obtaining a high level of control over the homotopy type of the attaching map for the top-cell of the gyration, which we provide by expanding upon the method of \cite{ChenTher:gy_stab}*{Part I}. Thus the homotopy theoretic generalisation of Fico's Lemmata arises as in the Main Theorem below.

\begin{mainthm}
    Let \(n\geq5\) and \(k\geq 2\) be integers and \(\tau\in\pi_{k-1}(\mathrm{SO}(n))\). 
    \begin{enumerate}
        \item If \(M\) and \(N\) are path-connected \(n\)-dimensional Poincar\'e Duality complexes, both with a single \(n\)-cell, then if \(k \leq n-3\) there is a homotopy equivalence 
        \[
            \mathcal{G}_\tau^k(M\#N)\simeq\mathcal{G}_\tau^k(M)\#\mathcal{G}_\tau^k(N).
        \]
        \item If \(q \geq p \geq 2\) with \(q\geq3\) and \(n=p+q\), then if \(k\leq q-1\) there is a homotopy equivalence
        \[
            \mathcal{G}_\tau^k(S^p \times S^q) \simeq (S^p \times S^{q+k-1})\#(S^q \times S^{p+k-1}).
        \]
    \end{enumerate}
\end{mainthm}

Each part is proved separately in Theorems \ref{thm:gy_connsums} and \ref{thm:gy_sphereprods'}, respectively. Indeed, this paper is organised as follows. Section \ref{sec:prelims} contains preliminary theory and in Section \ref{sec:gyrations} we obtain the desired control over gyrations and the attaching maps of their top-cells. This is first put to use in Section \ref{sec:conn_sums}, which contains the argument for gyrations of connected sums. Section \ref{sec:sphere_prods} is longer and contains the argument for gyrations of sphere products. We close with a discussion of applications in Section \ref{sec:apps}, with short sections on the implications for \textit{iterated gyrations}, \textit{gyration stability} and \textit{inertness}.

\subsubsection*{Acknowledgements} The author was supported by the Heilbronn Institute for Mathematical Research during preparation of this work. He also wishes to thank Lewis Stanton for his comments on an earlier draft of this paper, as well as Stephen Theriault for fruitful discussions that lead to this project. The author also thanks the anonymous reviewer for their insightful comments.

\section{Preliminaries}
\label{sec:prelims}

Here we give some preliminaries common to the following sections of this paper. 

\subsection{Notation}\label{subsec:notation}

First and foremost, this paper works with maps between wedges of spaces in great detail, so for the sake of clarity we set up the following notation before beginning in earnest. Given based maps \(f:A\rightarrow X\) and \(g:B\rightarrow Y\) we define the \textit{wedge} of \(f\) and \(g\) to be the map \[f\vee g:A\vee B\longrightarrow X \vee Y\] which is to say, \(f\) on the first summand and \(g\) on the second. Furthermore, if \(Y=X\) then we may define the \textit{wedge sum} of \(f\) and \(g\) to be the composite \[f\perp g: A\vee B\xlongrightarrow{f\vee g} X\vee X\xlongrightarrow{\nabla} X\] where \(\nabla\) denotes the fold map. If in addition we have \(A=B\) and \(A\) is a co-\(H\)-space with comultiplication~\(\sigma\), then the \textit{sum} of \(f\) and \(g\) is the composite \[f+g:A\xlongrightarrow{\sigma} A\vee A\xlongrightarrow{f\perp g} X.\]

Let \(M\) be a path-connected \(n\)-dimensional Poincar\'e Duality complex which has a \(CW\)-structure with a single \(n\)-cell; fix such a \(CW\)-structure. Note that simply-connected Poincar\'e Duality complexes always have this property. Let \(\overline{M}\) be the \((n-1)\)-skeleton  of \(M\). Then there is a homotopy cofibration 
\[
    S^{n-1}\xrightarrow{f_{M}}\overline{M}\xrightarrow{i_{M}} M
\] 
where \(f_{M}\) is the attaching map for the \(n\)-cell of \(M\) and \(i_{M}\) is the inclusion of 
the \((n-1)\)-skeleton. If \(N\) is another such Poincar\'{e} Duality complex then the connected sum \(M\conn N\) can be formed in the usual way by deleting \(n\)-disks and gluing along boundary spheres. Observe however that the connected sum is also exhibited by the homotopy cofibration 
\begin{equation} 
   \label{eq:connsumcofib} 
   S^{n-1}\xrightarrow{f_{M}\check{+}f_{M}}\overline{M}\vee\overline{N}\longrightarrow M\conn N 
\end{equation} 
where $f_{M}\check{+}f_{N}$ is the composite 
\[
    f_{M}\check{+}f_{N}\colon S^{n-1}\xlongrightarrow{\sigma} S^{n-1}\vee S^{n-1}\xrightarrow{f_{M}\vee f_{N}} 
    \overline{M}\vee\overline{N}
\] 
with $\sigma$ being the standard comultiplication on the sphere \(S^{n-1}\). The homotopy cofibration (\ref{eq:connsumcofib}) may be rephrased, in that \(\overline{M\#N}\simeq\overline{M}\vee\overline{N}\) and the attaching map of the \(n\)-cell of \(M \# N\) is homotopic to \(f_M \check{+} f_N\). Note that if \(M\) and \(N\) have the homotopy type of oriented manifolds then this construction agrees up to homotopy with the usual orientation preserving connected sum.

\subsection{A Deviation Map and Whitehead products}\label{subsec:dev}
This subsection contains some material covered in \cite{ChenTher:gy_stab}*{Section 1}, where the notion of a deviation map was introduced in the following setting. Let \(A\) and \(B\) be path-connected spaces. The \emph{right half-smash} is the quotient space \(A\amsrtimes B=(A\times B)/\sim\) obtained by collapsing \(B\) to the basepoint, and indeed, there is a homotopy cofibration 
\[
    B\xrightarrow{j_2} A\times B\xrightarrow{q} A\amsrtimes B
\] 
where \(j_2\) denotes the inclusion of the second factor into the product. If \(A\) is co-\(H\)-space then \(A\amsrtimes B\) is a co-\(H\)-space and there is a homotopy equivalence
\begin{equation}\label{eq:halfsmash_equivalence}
    e:A\amsrtimes B\rightarrow A\vee(A\wedge B)
\end{equation}
which is natural for co-\(H\)-maps \(A \rightarrow C\) between simply-connected co-$H$-spaces \cite{ChenTher:gy_stab}*{Lemma 1.1}.

Suppose now that we have two simply-connected co-\(H\)-spaces \(A\) and \(C\) with a map \(f:A \rightarrow C\), and letting \(\mathbb{1}:B\rightarrow B\) be an identity map, consider the diagram 
\begin{equation}\label{devdgrm} 
    \begin{tikzcd}[row sep=3em, column sep = 3em]
        A\amsrtimes B \arrow[d, "f\amsrtimes \mathbb{1}"] \arrow[r, "e"] & A\vee(A\wedge B) \arrow[d, "f\vee(f\wedge \mathbb{1})"] \\
        C\amsrtimes B \arrow[r, "e"] & C\vee(C\wedge B).
    \end{tikzcd}
\end{equation}
If \(f\) is a co-\(H\)-map then the naturality of the half-smash implies that it homotopy commutes. On the other hand, if \(f\) is not a co-\(H\)-map then the diagram may not homotopy commute, so analogously to the co-\(H\)-deviation of a map, define
\[
    \delta:=e\circ(f\amsrtimes \mathbb{1})-((f\vee(f\wedge \mathbb{1}))\circ e):A\amsrtimes B \rightarrow C\vee (C\wedge B)
\] 
to be the difference of the two directions around the diagram. Clearly \(\delta\simeq\ast\) if and only if (\ref{devdgrm}) homotopy commutes. Of note for us is the following result from \cite{ChenTher:gy_stab}, where we let \(I\) denote the inclusion of the wedge into the product.

\begin{lem}[\cite{ChenTher:gy_stab}*{Lemma 1.2}]\label{lem:deltacomp_null}
    The composition \(A\amsrtimes B \xrightarrow{\delta} C\vee(C\wedge B)\xrightarrow{I} C\times(C\wedge B)\) is null homotopic. \qed
\end{lem}

For our context, Lemma \ref{lem:deltacomp_null} is vital because it allows us to introduce Whitehead products to the picture - this is a crucial step for the work with gyrations that will follow in this paper. Define maps \(ev_{1}\) and \(ev_{2}\) by the composites 
\begin{equation*}
    ev_1:\Sigma\Omega X \xrightarrow{ev} X \xrightarrow{i_1} X\vee Y \text{\; and\;}
    ev_2:\Sigma\Omega Y \xrightarrow{ev} Y \xrightarrow{i_2} X\vee Y
\end{equation*} 
where \(ev\) is the usual evaluation map. Ganea \cite{ganea65} showed that there is a homotopy fibration 
\begin{equation*}
    \Omega X\ast\Omega Y \xrightarrow{[ev_1,ev_2]} X\vee Y \xrightarrow{I} X\times Y
\end{equation*}
where $I$ is the inclusion of the wedge into the product and $[ev_{1},ev_{2}]$ is the Whitehead product of~$ev_{1}$ and $ev_{2}$. Thus the null homotopy of Lemma \ref{lem:deltacomp_null} implies that there exists a lift \(\psi\) as in the following diagram
\begin{equation}\label{dgm:general_lift}
    \begin{tikzcd}[row sep=3em, column sep = 3em]
            & \Omega C\ast\Omega(C \wedge B) \arrow[d, "{[ev_1,ev_2]}"] \\
            A \amsrtimes B \arrow[r, "\delta"] \arrow[ur, dashed, "\psi"] & C\vee(C \wedge B). 
    \end{tikzcd}
\end{equation}

We wish to obtain more control in the special case when the map \(f\) in the deviation \(\delta\) is a Whitehead product, and how this interacts with certain pinch maps. In what follows, given a wedge sum \(\bigvee X_r\) let \(p_i:\bigvee X_r\rightarrow X_i\) denote the pinch map to the \(i^\text{th}\) summand, and let \(p_{i,j}:\bigvee X_r\rightarrow X_i\vee X_j\) denote the pinch map to the sub-wedge of the \(i^\text{th}\) and \(j^\text{th}\) summands.

Now we specialise. Let \(X\) and \(Y\) be two path connected spaces and take in an integer \(k\geq2\). Then, in the context of (\ref{dgm:general_lift}), take \(A=\Sigma X \wedge Y\), \(B=S^{k-1}\) and \(C=\Sigma X \vee \Sigma Y\). Write \(\delta_w\) for the deviation of (\ref{devdgrm}) when the map \(f\) is the Whitehead product \(w:\Sigma X \wedge Y\rightarrow \Sigma X \vee \Sigma Y\) of the two inclusions into the wedge. By definition, then, we have
\begin{equation}\label{eq:whitehead_deviation}
    \delta_w:(\Sigma X \wedge Y)\amsrtimes S^{k-1} \longrightarrow \Sigma X \vee \Sigma Y\vee\Sigma^k X \vee \Sigma^k Y.
\end{equation}

\begin{lem}\label{lem:deviation_pinch_null}
     The composites \(p_{1,3}\circ\delta_w\) and \(p_{2,4}\circ\delta_w\) are null homotopic.
\end{lem}

\begin{proof}
    Take \(p_{1,3}\circ\delta_w\) first,  and consider the following diagram
    \begin{equation}\label{dgm:special_deviation}
        \begin{tikzcd}[row sep=3em, column sep = 3em]
            (\Sigma X \wedge Y)\amsrtimes S^{k-1} \arrow[d, "w\amsrtimes \mathbb{1}"] \arrow[r, "e"] & (\Sigma X \wedge Y)\vee(\Sigma^k X \wedge Y) \arrow[d, "w\vee \Sigma^{k-1}w"] \\
            (\Sigma X \vee \Sigma Y)\amsrtimes S^{k-1} \arrow[r, "e"] \arrow[d, "p_1\amsrtimes\mathbb{1}"] & \Sigma X \vee \Sigma Y\vee\Sigma^k X \vee \Sigma^k Y \arrow[d, "p_{1,3}"] \\ \Sigma X \amsrtimes S^{k-1} \arrow[r, "e"] & \Sigma X \vee \Sigma^k X.
        \end{tikzcd}
    \end{equation}
    The upper square is Diagram (\ref{devdgrm}) in our specialised context, and the lower square comes via naturality of the half-smash with respect to pinch maps. The diagram therefore commutes up to homotopy. 

     It is a classical fact the Whitehead product \(w\) is in the kernel of the suspension map, and that it is null when composed with a pinch map to a single summand. This implies that both vertical composites of (\ref{dgm:special_deviation}) are null homotopic. Thus the definition of \(\delta_w\) together with homotopy commutativity of (\ref{dgm:special_deviation}) implies that \(p_{1,3}\circ\delta_w\simeq\ast\), as asserted. For the null homotopy of \(p_{2,4}\circ\delta_w\), replace the lower square of (\ref{dgm:special_deviation}) with the homotopy commutative diagram
    \begin{equation*}
        \begin{tikzcd}[row sep=3em, column sep = 3em]
            (\Sigma X \vee \Sigma Y)\amsrtimes S^{k-1} \arrow[r, "e"] \arrow[d, "p_2\amsrtimes\mathbb{1}"] & \Sigma X \vee \Sigma Y\vee\Sigma^k X \vee \Sigma^k Y \arrow[d, "p_{2,4}"] \\ \Sigma Y \amsrtimes S^{k-1} \arrow[r, "e"] & \Sigma Y \vee \Sigma^k Y.
        \end{tikzcd}
    \end{equation*}
    The assertion then follows from identical reasoning. 
\end{proof}

In our specialised situation, observe that the homotopy fibre \(\Omega C \ast \Omega (C\wedge B)\) of Diagram (\ref{dgm:general_lift}) is homotopy equivalent to the following
\[
    \Omega C \ast \Omega (C\wedge B) \simeq \Sigma\Omega\Sigma(X \vee Y)\wedge\Omega \Sigma^k (X \vee Y).
\]
We can decompose this further using the James construction. Indeed, letting
\[
    W=\Sigma\left(\left(\bigvee_{r\geq1}\Big(X \vee Y\Big)^{\wedge r}\right)\wedge\left(\bigvee_{\ell\geq1}\Big(\Sigma^{k-1} X \vee \Sigma^{k-1} Y\Big)^{\wedge \ell}\right) \right)
\]
we have \(\Omega C \ast \Omega (C\wedge B) \simeq W\). Note that since the smash product exhibits distributivity over the wedge sum, \(W\) may be decomposed into an infinite wedge of smash products of \(X\) and \(Y\) (and suspensions thereof). We will use Lemma \ref{lem:deviation_pinch_null} to determine the behaviour of the lift from (\ref{dgm:general_lift}) when restricted to certain sub-wedges. Indeed, define
\begin{equation}\label{eq:sub-wedges}
    W_X=\bigvee_{r,\ell\geq1} \Sigma\Big(X^{\wedge r}\wedge(\Sigma^{k-1} X)^{\wedge\ell}\Big) \text{\; and \;} W_Y=\bigvee_{r,\ell\geq1} \Sigma\Big(Y^{\wedge r}\wedge(\Sigma^{k-1} Y)^{\wedge\ell}\Big).
\end{equation}
Both \(W_X\) and \(W_Y\) occur as sub-wedges in the expansion of the smash product in \(W\), and hence there are pinch maps \(p_X:W\rightarrow W_X\) and \(p_Y:W\rightarrow W_Y\). 

\begin{prop}\label{prop:psi_pinch}
    There is a homotopy commutative diagram defining a lift \(\psi_w\) as follows
    \begin{equation}\label{dgm:lift}
        \begin{tikzcd}[row sep=3em, column sep = 3em]
            & W \arrow[d, "{[ev_1,ev_2]}"] \\
            (\Sigma X \wedge Y)\amsrtimes S^{k-1} \arrow[r, swap, "\delta_w"] \arrow[ur, dashed, "\psi_w"] & \Sigma X \vee \Sigma Y\vee\Sigma^k X \vee \Sigma^k Y
        \end{tikzcd}
    \end{equation}
    such the composites \(p_X\circ\psi_w\) and \(p_Y\circ\psi_w\) are null homotopic. 
\end{prop}

\begin{proof}
    The asserted diagram exists by application of (\ref{dgm:general_lift}) to our special case - we label the lift as \(\psi_w\) to distinguish from the general case. First consider \(p_X\circ\psi_w\). By \cite{ganea65} and the James Construction, the homotopy fibre of the map \(I_X:\Sigma X \vee \Sigma^k X\rightarrow \Sigma X \times \Sigma^k X\) is homotopy equivalent to the wedge \(W_X\). We therefore have the following homotopy commutative diagram, in which both columns are homotopy fibrations
    \begin{equation}\label{dgm:fibre_factor_P}
        \begin{tikzcd}[row sep=3em, column sep = 3em]
            & W \arrow[d] \arrow[r, swap, "p_X"] & W_X \arrow[d] \\
            (\Sigma X \wedge Y)\amsrtimes S^{k-1} \arrow[r, swap, "\delta_w"] \arrow[ur, dashed, "\psi_w"] & \Sigma X \vee \Sigma Y\vee\Sigma^k X \vee \Sigma^k Y \arrow[d, "I"] \arrow[r, swap, "p_{1,3}"] & \Sigma X \vee \Sigma^k X \arrow[d, "I_X"] \\
            & (\Sigma X \vee \Sigma Y) \times (\Sigma^k X \vee \Sigma^k Y) \arrow[r, swap, "p_1 \times p_1"] & \Sigma X \times \Sigma^k X.
        \end{tikzcd}
    \end{equation}
    By Lemma \ref{lem:deviation_pinch_null} there is a homotopy \(p_{1,3}\circ\delta_w\simeq\ast\), so \(p_{1,3}\circ\delta_w\) is lifted to the homotopy fibre of \(I_X\) by a constant map. Lifts are unique up to homotopy, so since the composite \(p_X\circ\psi_w\) is another lift of \(p_{1,3}\circ\delta_w\), we must have \(p_X\circ\psi_w\simeq\ast\). The same conclusion follows for \(p_Y\circ\psi_w\) by the evident symmetric argument.
\end{proof}

Proposition \ref{prop:psi_pinch} enables us give a detailed decomposition of the attaching map of the top-cell of the gyration \(\mathcal{G}_\tau^k(S^p\times S^q)\). We postpone this until Section \ref{sec:sphere_prods}, in favour of giving a proper discussion of the homotopy theory of gyrations.

\section{Gyrations and Manipulations of their Attaching Maps}\label{sec:gyrations}

We first recall the definition of gyrations from the Introduction. Let \(M\) be a Poincar\'e Duality complex of dimension \(n\) with a single \(n\)-cell; from our notation from Subsection \ref{subsec:notation}, we write $\overline{M}$ for the $(n-1)$-skeleton of $M$ and $f_M$ for the attaching map of the top-cell. Now take $k\geq 2$ be an integer and take a homotopy class \(\tau\in\pi_{k-1}(\mathrm{SO}(n))\), then using the standard linear action of \(\mathrm{SO}(n)\) on \(S^{n-1}\) define the map 
\begin{gather*}
    t: S^{n-1}\times S^{k-1} \rightarrow  S^{n-1}\times S^{k-1} \\
     (a,x)  \mapsto  (\tau(x)\cdot a,x)
\end{gather*}

\begin{defn}\label{def:gy}
    Let \(k\geq2\) be an integer and let \(M\) be Poincar\'e Duality complex of dimension \(n\) with a single \(n\)-cell. Define the \textit{\(k\)-gyration of \(M\) by \(\tau\)} to be the space defined by the (strict) pushout
        \begin{equation*}
            \begin{tikzcd}[row sep=3em, column sep=3em]
                S^{n-1}\times S^{k-1} \arrow[r, "\mathbb{1}\times \iota"] \arrow[d, "(f_M\times \mathbb{1})\circ t"] & S^{n-1}\times D^k \arrow[d] \\
                \overline{M}\times S^{k-1} \arrow[r] & \mathcal{G}^k_\tau(M) 
            \end{tikzcd}
        \end{equation*} 
    where $\iota$ is the inclusion of the boundary of the disc and \(\mathbb{1}\) denotes an identity map. When the context is clear, we will usually just write \textit{gyration} for \(\mathcal{G}^k_\tau(M)\). 
\end{defn}

We call \(k\) the \textit{index} of the gyration and \(\tau\) the \textit{twisting}. When \(\tau\) is trivial the map \(t\) is homotopic to the identity map; we call the resulting gyration the \textit{trivial \(k\)-gyration} and write it as \(\mathcal{G}^k_0(M)\). Moreover, in the non-trivial case the action of \(\tau\) preserves orientation, so it follows that if \(M\) has the homotopy type of an oriented manifold then \(\mathcal{G}_\tau^k(M)\) is an $(n+k-1)$-manifold with an orientation inherited from that of \(M\). This is also clear from the definition of gyrations via surgery - details are found in work of Huang \cite{huang_inertness24}*{Section 12}.

We wish to determine further details about the homotopy type of the attaching map of a gyration. To do this, let us first consider a map \(t':S^{n-1}\amsrtimes S^{k-1} \rightarrow S^{n-1}\amsrtimes S^{k-1}\) induced by the cofibration diagram
    \begin{equation} \label{dgm:t'def} 
        \begin{tikzcd}[row sep=3em, column sep = 3em]
            S^{k-1} \arrow[d, equal] \arrow[r, "{j_2}"] & S^{n-1}\times S^{k-1} \arrow[d, "t"] \arrow[r] & S^{n-1}\amsrtimes S^{k-1} \arrow[d, dashed, "t'"] \\
            S^{k-1} \arrow[r, "{j_2}"] & S^{n-1}\times S^{k-1} \arrow[r] & S^{n-1}\amsrtimes S^{k-1}.
        \end{tikzcd}
    \end{equation}

\begin{lem}\label{lem:t'}
    The map \(t'\) of (\ref{dgm:t'def}) is such that
    \begin{enumerate}
        \item[(i)] \(t'\) is a homotopy equivalence;
        \item[(ii)] the composite 
        \[
            \phi_{\tau}: S^{n+k-2} \xrightarrow{j} S^{n-1}\amsrtimes S^{k-1} \xrightarrow{t'} S^{n-1}\amsrtimes S^{k-1} \xrightarrow{f_M\amsrtimes \mathbb{1}} \overline{M}\amsrtimes S^{k-1}
        \]
        is homotopic to the attaching map of the top-cell of \(\mathcal{G}^{k}_{\tau}(M)\).
    \end{enumerate}
\end{lem}

\begin{proof}
    Applying the Five Lemma to the homology of the spaces in the cofibration sequence derived from (\ref{dgm:t'def}) implies that \(t'\) induces an isomorphism on homology. Since \(S^{n-1}\amsrtimes S^{k-1}\) is simply-connected, Whitehead's Theorem then implies that \(t'\) is a homotopy equivalence. Part (ii) is the statement of \cite{ChenTher:gy_stab}*{Lemma 3.2}.
\end{proof}

In general, let 
\(
    i:A \rightarrow A \amsrtimes B\) and \(\pi:A \amsrtimes B \rightarrow  A
\)
be the canonical inclusion of the first factor and the natural projection from the first factor of the half-smash, respectively. Define \(\overline{\tau}\) to be the composite 
    \begin{equation}\label{def_taubar} 
        \overline{\tau}\colon S^{n+k-2} \xrightarrow{j} S^{n-1}\amsrtimes S^{k-1} \xrightarrow{t'} S^{n-1}\amsrtimes S^{k-1} \xrightarrow{\pi}  S^{n-1}
    \end{equation}
and recall the classical \(J\)-homomorphism \(J:\pi_r(\mathrm{SO}(n))\rightarrow\pi_{n+r-1}(S^n)\). By abuse of notation, we also write \(J:\pi_r(\mathrm{SO})\rightarrow\pi_r^S\) for the stable formulation of the \(J\)-homomorphism.

\begin{lem}\label{lem:t'_comp_j} 
    If \(2 \leq k \leq n-2\), then \(t'\circ j\simeq j + i\circ\overline{\tau}\). Moreover, for \(k\) in this range, \(\overline{\tau}\) is represented by a stable class \(\widetilde{\tau}\in im(J)\subseteq\pi_{k-1}^S\). 
\end{lem} 

\begin{proof}
    The statement that \(t'\circ j\simeq j + i\circ\overline{\tau}\) is from \cite{ChenTher:gy_stab}*{Lemma 3.9(ii)}. The second holds since for \(k\) in this range the homotopy group \(\pi_{n+k-2}(S^{n-1})\) is stable. By \cite{ChenTher:gy_stab}*{Proposition 3.10} there is a homotopy \(\Sigma\overline{\tau}\simeq J(\tau)\), which then implies that for \(k\) in this range \(\overline{\tau}\) is represented by some class in the image of the stable \(J\)-homomorphism.
\end{proof}

The following corollary is crucial for our arguments in Section \ref{sec:conn_sums}, where we decompose the gyration of a connected sum. 

\begin{cor}\label{cor:t'_coH}
    If \(2 \leq k \leq n-3\) then \(t'\) is a co-\(H\)-map.
\end{cor}

\begin{proof}
    The equivalence \(t'\) induces another homotopy equivalence \(t'':S^{n-1} \vee S^{n+k-2} \rightarrow S^{n-1} \vee S^{n+k-2}\) as in the diagram
     \begin{equation*} 
        \begin{tikzcd}[row sep=3em, column sep=3em]
            S^{n-1}\amsrtimes S^{k-1} \arrow[r, "t'"] \arrow[d, "e"] & S^{n-1}\amsrtimes S^{k-1} \arrow[d, "e"] \\
            S^{n-1}\vee S^{n+k-2} \arrow[r, dashed, "t''"] & S^{n-1}\vee S^{n+k-2} 
        \end{tikzcd}
    \end{equation*}
    such that \(t'' \circ i_1 \simeq i_1\) and \(t'' \circ i_2 \simeq i_2 + i_1\circ\overline{\tau}\). For \(t'\) to be a co-\(H\)-map it is therefore sufficient to check that \(\overline{\tau}\) is a co-\(H\)-map. For \(k\) in this range, the stable class \(\widetilde{\tau}\) of Lemma \ref{lem:t'_comp_j} representing \(\overline{\tau}\) exists unstably as a class \(\widetilde{\tau}_{n-2}\in\pi_{n+k-3}(S^{n-2})\) such that \(\overline{\tau}\simeq\Sigma\widetilde{\tau}_{n-2}\). Thus \(\overline{\tau}\) is a suspension, and so a co-\(H\)-map.
\end{proof}

For our arguments in Sections \ref{sec:conn_sums} and \ref{sec:sphere_prods} it is essential to give more detail on the homotopy class of the attaching map of the top-cell of a gyration, and indeed, we will also specialise to a particular case when \(f_M\) vanishes after iterated suspensions. Combining Lemma \ref{lem:t'}(ii) and Lemma \ref{lem:t'_comp_j} implies that the attaching map \(\phi_\tau\) for the top-cell of a gyration may be rewritten via the homotopies
\[
    \begin{split}
    \phi_\tau = &~ (f_M\amsrtimes\mathbb{1})\circ t'\circ j\\
    \simeq &~  (f_M\amsrtimes\mathbb{1})\circ j + (f_M\amsrtimes\mathbb{1})\circ i \circ \overline{\tau}\\
    \simeq &~  (f_M\amsrtimes\mathbb{1})\circ j + i \circ f_M\circ \overline{\tau}.
    \end{split}
\]
where the last equivalence comes from noting that the inclusion map \(i\) is natural with respect to half-smashes of maps. If \(\overline{M}\) is a co-\(H\)-space, say \(M\simeq\Sigma X\) for some path-connected space \(X\), then there is a homotopy equivalence
\[
    e:\overline{M}\amsrtimes S^{k-1} \xlongrightarrow{\simeq} \overline{M}\vee\Sigma^{k-1}\overline{M}
\]
as in (\ref{eq:halfsmash_equivalence}). Writing \(\varphi_\tau=e\circ\phi_\tau:S^{n+k-2}\rightarrow\overline{M}\vee\Sigma^{k-1}\overline{M}\) one obtains an equivalent attaching map 
\begin{equation}\label{eq:general_varphi}
    \varphi_\tau\simeq e\circ (f_M\amsrtimes\mathbb{1})\circ j + i_1\circ f_M\circ \overline{\tau}.
\end{equation} 
Moreover, recall the deviation \(\delta\) of (\ref{devdgrm}) and note that by definition, in our context we have a homotopy 
\[
    \delta\simeq e\circ (f_M\amsrtimes\mathbb{1}) - (f_M\vee\Sigma^{k-1}f_M)\circ e
\]
which enables the following statement.

\begin{prop}\label{prop:varphi_for_M}
    There is a homotopy equivalence
    \[
        \varphi_\tau\simeq i_2\circ(\Sigma^{k-1}f_M) + \delta\circ j + i_1\circ f_M\circ \overline{\tau}
    \]
    In particular, if \(\Sigma^{k-1}f_M\simeq\ast\) then \(\varphi_\tau\simeq \delta\circ j + i_1\circ f_M\circ \overline{\tau}\).
\end{prop}

\begin{proof}
    Consider the composite \(e\circ(f_M\amsrtimes\mathbb{1})\circ j\). Rearranging the identity for the deviation map \(\delta\) and precomposing with \(j\) we have a homotopy 
    \[
        e\circ (f_M\amsrtimes\mathbb{1})\circ j\simeq\delta\circ j + (f_M\vee\Sigma^{k-1}f_M)\circ i_2.
    \]
    Naturality of inclusion maps gives a further homotopy \((f_M\vee\Sigma^{k-1}f_M)\circ i_2\simeq i_2\circ\Sigma^{k-1}f_M\). Thus (\ref{eq:general_varphi}) becomes
    \begin{equation}\label{eq:phi_tau'}
        \varphi_\tau\simeq i_2\circ(\Sigma^{k-1}f_M)+\delta\circ j + i_1\circ f_M\circ \overline{\tau}.
    \end{equation}
    The statement when \(\Sigma^{k-1}f_M\) is null homotopic follows, as it forces the first term of (\ref{eq:phi_tau'}) to vanish. 
\end{proof}

\begin{rem}
    Note the analogy between this more general decomposition and \cite{ChenTher:gy_stab}*{Corollary 3.14}, where the homotopy type of \(\varphi_\tau\) was found in similar terms when \(M\) is a complex, quaternionic or octonionic projective plane.
\end{rem}

\begin{rem}\label{rem:varphi_0'}
    The only part of Proposition \ref{prop:varphi_for_M} that depends on the twisting \(\tau\) is the last summand, \(i_1\circ f_M\circ\overline{\tau}\). Indeed, for the trivial gyration the class \(\overline{\tau}\) vanishes, so we have
    \[
        \varphi_0\simeq i_2\circ(\Sigma^{k-1}f_M) + \delta\circ j.
    \]
\end{rem}

The upshot of this work is that the problem of identifying the homotopy class of \(\varphi_\tau\) is in large part reduced to determining the homotopy class of the composite \(\delta\circ j\). In general this is an abstract problem, but there are certain gyrations for which it can be determined. This is the focus of Section \ref{sec:sphere_prods}, where it is done for gyrations of sphere products.

\section{Gyrations of Connected Sums}
\label{sec:conn_sums}

In this section we give an argument for when gyrations of connected sums are homotopic to connected sums of gyrations; that is, when we have \(\mathcal{G}_\tau^k(M\#N) \simeq \mathcal{G}_\tau^k(M) \# \mathcal{G}_\tau^k(N)\). For clarity, in what follows we will denote the relevant attaching maps by 
\begin{gather*}
    \phi_\tau^{M}:S^{n+k-2}\rightarrow \overline{M}\amsrtimes S^{k-1} \text{, \;}
    \phi_\tau^{N}:S^{n+k-2}\rightarrow \overline{N}\amsrtimes S^{k-1} \\ \text{\; and \;}
    \phi_\tau^{M\#N}:S^{n+k-2}\rightarrow (\overline{M}\vee\overline{N})\amsrtimes S^{k-1}.
\end{gather*}
As per Subsection \ref{subsec:notation}, showing that \(\mathcal{G}_\tau^k(M\#N) \simeq \mathcal{G}_\tau^k(M) \# \mathcal{G}_\tau^k(N)\) amounts to proving that \(\phi^{M\#N}_\tau\) is homotopic to the sum \(\phi^M_\tau\check{+}\phi^N_\tau\). 

Recall that the attaching map for the top-cell of \(M \# N\) is \(f_M \check{+} f_N:S^{n-1}\rightarrow \overline{M}\vee\overline{N}\), so by Lemma \ref{lem:t'}(ii) the map \(\phi_\tau^{M\#N}\) is given by the composite 
\[
    \phi_\tau^{M\#N}\simeq((f_M\check{+}f_N)\amsrtimes\mathbb{1})\circ t'\circ j.
\]
By definition, there is a homotopy \(f_M \check{+} f_N\simeq (f_M \vee f_N)\circ\sigma_{n-1}\) where \(\sigma_{n-1}\) denotes the usual comultiplication on \(S^{n-1}\). Recall also that the half-smash \(S^{n-1}\amsrtimes S^{k-1}\) has its own comulitplication, given by the map \(\sigma_{n-1}\amsrtimes\mathbb{1}\), and that for maps \(\alpha:X\rightarrow Y\), \(\beta:A\rightarrow X\) and \(\mathbb{1}:Z\rightarrow Z\) there is a homotopy
\begin{equation}\label{eq:halfsmash_of_comps}
    (\alpha \circ \beta)\amsrtimes\mathbb{1}\simeq(\alpha \amsrtimes\mathbb{1}) \circ (\beta\amsrtimes\mathbb{1})
\end{equation}

\begin{thm}\label{thm:gy_connsums}
    Let \(n\geq5\) and \(M\) and \(N\) be path-connected \(n\)-dimensional Poincar\'e Duality complexes, both with a single \(n\)-cell. Then for all indices \(2\leq k \leq n-3\) and all twistings \(\tau\in\pi_{k-1}(\mathrm{SO}(n))\) there is a homotopy equivalence 
    \[
        \mathcal{G}_\tau^k(M\#N)\simeq\mathcal{G}_\tau^k(M)\#\mathcal{G}_\tau^k(N).
    \]
\end{thm}

\begin{proof}
By (\ref{eq:halfsmash_of_comps}) there is a homotopy \((f_M \check{+} f_N)\amsrtimes\mathbb{1}\simeq((f_M \vee f_N)\amsrtimes\mathbb{1})\circ(\sigma_{n-1}\amsrtimes\mathbb{1})\). Moreover, since there is in general a homeomorphism \((X \vee Y)\amsrtimes Z \cong (X\amsrtimes Z)\vee(Y\amsrtimes Z)\) we see that \((f_M \vee f_N)\amsrtimes\mathbb{1}\simeq( f_M\amsrtimes\mathbb{1})\vee( f_N\amsrtimes\mathbb{1})\). Using the defining property of a co-\(H\)-map, there is a homotopy \((\sigma_{n-1}\amsrtimes\mathbb{1})\circ t'\simeq(t'\vee t')\circ(\sigma_{n-1}\amsrtimes\mathbb{1})\) by virtue of Corollary \ref{cor:t'_coH}. Similarly, the map \(j\) is a co-\(H\)-map by \cite{ChenTher:gy_stab}*{Lemma 2.1}, so one also has \((\sigma_{n-1}\amsrtimes\mathbb{1})\circ j\simeq(j\vee j)\circ(\sigma_{n-1}\amsrtimes\mathbb{1})\). Therefore
\[
    \begin{split}
    \phi^{M\#N}_\tau = &~ ((f_M\check{+}f_N)\amsrtimes\mathbb{1})\circ t'\circ j \\
    \simeq &~  ((f_M\amsrtimes\mathbb{1})\vee( f_N\amsrtimes\mathbb{1}))\circ(\sigma_{n-1}\amsrtimes\mathbb{1})\circ t'\circ j \\
    \simeq &~ ((f_M\amsrtimes\mathbb{1})\vee( f_N\amsrtimes\mathbb{1}))\circ(t'\vee t')\circ(\sigma_{n-1}\amsrtimes\mathbb{1})\circ j \\
    \simeq &~ ((f_M\amsrtimes\mathbb{1})\vee( f_N\amsrtimes\mathbb{1}))\circ(t'\vee t')\circ (j \vee j)\circ(\sigma_{n-1}\amsrtimes\mathbb{1}) \\
    \simeq &~ \big[((f_M\amsrtimes\mathbb{1})\circ t' \circ j)\vee((f_N\amsrtimes\mathbb{1}))\circ t' \circ j)\big]\circ(\sigma_{n-1}\amsrtimes\mathbb{1}) \\
    \simeq &~ (\phi^M_\tau \vee \phi^N_\tau) \circ(\sigma_{n-1}\amsrtimes\mathbb{1}) \\
    \simeq &~ \phi^M_\tau \check{+} \phi^N_\tau.
    \qedhere
    \end{split}
\]
\end{proof}

\section{Gyrations of Sphere Products}
\label{sec:sphere_prods}

Throughout this section let \(p\) and \(q\) be integers such that \(q \geq p \geq 2\) with \(q\geq3\), and let \(n=p+q\). We wish to determine the homotopy type of the gyration \(\mathcal{G}_\tau^k(S^p \times S^q)\). To begin, let \(w\) denote the Whitehead product of inclusions \(w:S^{n-1} \rightarrow S^{p} \vee S^{q}\), which attaches the top-dimensional cell to the sphere product \(S^p \times S^q\). Whitehead products suspend trivially, so \(\Sigma w\simeq\ast\), and thus by Proposition \ref{prop:varphi_for_M} the attaching map \(\varphi_\tau\) of the top-cell of \(\mathcal{G}_\tau^k(S^p \times S^q)\) is homotopic to
\begin{equation}\label{eq:sphereprodgyration_attachingmap}
    \varphi_\tau\simeq \delta_w\circ j + i_1\circ w\circ \overline{\tau}
\end{equation}
where \(\delta_w\) denotes the deviation of (\ref{eq:whitehead_deviation}) with \(X=S^{p-1}\) and \(Y=S^{q-1}\). Note also that regardless of \(\tau\), the \((n+k-2)\)-skeleton of \(\mathcal{G}_\tau^k(S^p \times S^q)\) is homotopic to 
\[
    (S^p \vee S^q) \amsrtimes S^{k-1} \simeq S^p \vee S^q \vee S^{p+k-1} \vee S^{q+k-1}
\]
from which the fact that \(\mathcal{G}_\tau^k(S^p \times S^q)\) is a Poincar\'e Duality complex immediately yields the following cohomology isomorphism.

\begin{lem}\label{lem:cohomology}
    There is a ring isomorphism \(H^*(\mathcal{G}_\tau^k(S^p \times S^q)) \cong H^*((S^p \times S^{q+k-1})\#(S^q \times S^{p+k-1}))\) \qed
\end{lem}

\begin{rem}
    A broader cohomological result is known when \(k=2\), concerning the notion of an \textit{algebraic gyration}; see \cite{fanwang}*{Proposition 4.9}.
\end{rem}

We wish to show that for indices \(k\) in a certain range this isomorphism arises from a homotopy equivalence, which we will perform by determining the homotopy class of the attaching map \(\varphi_\tau\), using the work of Section \ref{sec:gyrations}. We will proceed by first considering the composite \(w\circ\overline{\tau}\) and then \(\delta_w\circ j\) before reaching a combined result. 

From \cite{IriyeKishimoto_fatwedge}*{Lemma 6.7} we see that \(w\circ\overline{\tau}\) is null homotopic if and only if \(\overline{\tau}\) is null homotopic. Thus, a priori, (\ref{eq:sphereprodgyration_attachingmap}) indicates that for non-trivial twistings the attaching map \(\varphi_\tau\) differs from \(\varphi_0\) by the addition of a non-trivial composite. Before going on it is crucial to recall the following result concerning interactions between suspensions and Whitehead products..

\begin{lem}[\cite{whitehead-elements}*{Theorem X.8.18}]\label{lem:whitehead_prods_and_susps}
    Let \(\alpha\in\pi_{p+1}(X)\), \(\beta\in\pi_{q+1}(X), \gamma\in\pi_i(S^p)\text{ and }\delta\in\pi_j(S^q)\). Then 
    \[
    \pushQED{\qed} 
        [\alpha\circ\Sigma\gamma,\beta\circ\Sigma\delta] \simeq [\alpha,\beta]\circ\Sigma(\gamma\wedge\delta).
    \qedhere
    \popQED 
    \]
\end{lem} 

In what follows, when discussing the wedge \(S^p \vee S^q \vee S^{p+k-1} \vee S^{q+k-1}\) let \(i_p, i_q, i_{p+k-1}\text{ and }i_{q+k-1}\) denote the inclusions of each individual summand.

\begin{lem}\label{lem:taubar_whitehead}
    If \(2 \leq k\leq q-1\) then there exists a class \(\widetilde{\tau}\in\pi_{q+k-2}(S^{q-1})\) such that \(\Sigma^{p}\widetilde{\tau}\simeq\overline{\tau}\), and 
    \[
        i_1\circ w\circ\overline{\tau}\simeq[i_p,i_q\circ\Sigma\widetilde{\tau}].
    \]   
\end{lem}

\begin{proof}
    By Lemma \ref{lem:t'_comp_j} we have that \(\overline{\tau}\) is a stable class when \(k\leq n-2\). To infer the existence of the claimed class \(\widetilde{\tau}\) we need to find the rank at which \(\pi_{q+k-1}(S^q)\) is isomorphic to the stable group \(\pi_{k-1}^S\) after one suspension, so the inequality \(k\leq(n-2)-(p-1)=q-1\) is all that we require. Note also that there is a homotopy \(i_1\circ w\simeq[i_p,i_q]\). Letting \(\mathbb{1}_{X}\) denote the identity map on a space \(X\), one has \(\Sigma^{p}\widetilde{\tau}\simeq\Sigma(\mathbb{1}_{S^{p-1}}\wedge\widetilde{\tau})\). Thus
    \[
        i_1\circ w\circ\overline{\tau}\simeq [i_p,i_q]\circ\Sigma^p\widetilde{\tau}\simeq [i_p,i_q]\circ\Sigma(\mathbb{1}_{S^{p-1}}\wedge\widetilde{\tau})\simeq[i_p,i_q\circ\Sigma\widetilde{\tau}]
    \]
    by application of Lemma \ref{lem:whitehead_prods_and_susps}.
\end{proof}

Recalling Remark \ref{rem:varphi_0'}, Lemma \ref{lem:taubar_whitehead} determines the summand of \(\varphi_\tau\) that depends only on the twisting. That is, there is a homotopy
\[
    \varphi_\tau\simeq\varphi_0+[i_p,i_q\circ\Sigma\widetilde{\tau}]
\]
where \(\varphi_0\) denotes the attaching map of the top-cell of the trivial gyration \(\mathcal{G}_0^k(S^p \times S^q)\).

\begin{prop}
    There is a homotopy equivalence \(\varphi_0\simeq [i_p,i_{q+k-1}]+[i_q,i_{p+k-1}].\)
\end{prop}

\begin{proof}
    From (\ref{eq:sphereprodgyration_attachingmap}) we see that \(\varphi_0\) is homotopic to the composite \(\delta_w\circ j\). Taking Diagram (\ref{dgm:lift}), with \(X=S^{p-1}\) and \(Y=S^{q-1}\) and precomposing with \(j\), there exists a homotopy commutative diagram
    \begin{equation*}
        \begin{tikzcd}[row sep=3em, column sep = 3em]
            & W \arrow[d, "{[ev_1,ev_2]}"] \\
            S^{n+k-2} \arrow[r, swap, "\delta_w\circ j"] \arrow[ur, dashed, "\psi_w\circ j"] & S^p \vee S^q \vee S^{p+k-1} \vee S^{q+k-1}
        \end{tikzcd}
    \end{equation*} 
    where \(W\) denotes the space
    \[
        W=\Sigma\left(\left(\bigvee_{r\geq1}\Big(S^{p-1} \vee S^{q-1}\Big)^{\wedge r}\right)\wedge\left(\bigvee_{\ell\geq1}\Big(S^{p+k-2} \vee S^{q+k-2}\Big)^{\wedge \ell}\right) \right).
    \]
    
    Recall that there are the two sub-wedges of \(W\) from (\ref{eq:sub-wedges}), \(W_X\) and \(W_Y\), to which \(\psi_w\) restricts trivially by Proposition \ref{prop:psi_pinch}. With our \(X\) and \(Y\) these are now given by 
    \begin{equation}
        W_X=\bigvee_{r,\ell\geq1} S^{r(p-1)+\ell(p+k-2)+1} \text{\; and \;} W_Y=\bigvee_{r,\ell\geq1} S^{r(q-1)+\ell(q+k-2)+1}.
    \end{equation}
    The wedge \(W_X\) consists of exactly the smashes of copies of \(S^p\) (and suspensions thereof) that occur in the expansion of the smash product in \(W\), and likewise, \(W_Y\) consists of exactly those arising from copies of \(S^q\) and its suspensions. Furthermore, for connectivity reasons, the composite \(\psi_w \circ j\) vanishes on the spheres in \(W\) whose dimension exceeds \(n+k-2\). Therefore \(\psi_w\circ j\) is null homotopic on all but two spheres in the expansion of the smash product in \(W\), namely the two \(S^{n+k-2}\) summands that arise from the \(r=\ell=1\) case via the smash of the first and fourth, and second and third, summands in \(W\) (respectively). Thus, by the Hilton-Milnor Theorem, the composite \([ev_1,ev_2]\circ\psi_w\circ j\) is homotopic to - for some integers \(a\) and \(b\) - a sum of maps \(a\cdot[i_p,i_{q+k-1}]+b\cdot[i_q,i_{p+k-1}]\). Since these Whitehead products induce cup products in cohomology, Lemma \ref{lem:cohomology} implies that we may set \(a=b=1\) without loss of generality. Hence there is a chain of homotopies 
    \[
        \varphi_0\simeq\delta_w\circ j \simeq [ev_1,ev_2]\circ\psi_w\circ j \simeq [i_p,i_{q+k-1}]+[i_q,i_{p+k-1}].
        \qedhere
    \]
\end{proof}

Combining the results of this section, we have shown the following for the homotopy class of \(\varphi_\tau\).

\begin{prop}\label{prop:varphi_tau'}
    Let \(q \geq p \geq 2\) with \(q\geq3\) and \(n=p+q\). If \(2\leq k\leq q-1\) then there exists a homotopy class \(\widetilde{\tau}\in\pi_{q+k-2}(S^{q-1})\) such that there is a homotopy
    \[
    \pushQED{\qed} 
        \varphi_\tau \simeq [i_p,i_{q+k-1}] + [i_q,i_{p+k-1}] + [i_p,i_q\circ\Sigma\widetilde{\tau}].
    \qedhere
    \popQED 
    \]
\end{prop}

It is this decomposition of \(\varphi_\tau\) that gives the main result of this section, giving the homotopy theoretic connected sum decomposition for higher gyrations of sphere products. This directly generalises the \(k=2\) case originally formulated by Gonz\'{a}lez Acu\~{n}a (see also \cite{duan}*{Proposition 3.2}). In what follows, let
\[
    w_1:S^{n+k-2}\rightarrow S^{p}\vee S^{q+k-1}\text{\; and \;}w_2:S^{n+k-2}\rightarrow S^q\vee S^{p+k-1}
\]
denote the Whitehead Products attaching top-cells to the sphere products \(S^p \times S^{q+k-1}\) and \(S^q \times S^{p+k-1}\), respectively.

\begin{thm}\label{thm:gy_sphereprods'}
    Let \(q \geq p \geq 2\) with \(q\geq3\) and \(n=p+q\). Then for all indices \(2\leq k\leq q-1\) and all twistings \(\tau\in\pi_{k-1}(\mathrm{SO(n)})\) there is a homotopy equivalence
    \[
        \mathcal{G}_\tau^k(S^p \times S^q) \simeq (S^p \times S^{q+k-1})\#(S^q \times S^{p+k-1}).
    \]
\end{thm}

\begin{proof}
    Suppose there exists a homotopy equivalence 
    \[
        \varepsilon:S^p \vee S^q \vee S^{p+k-1}\vee S^{q+k-1}\xrightarrow{\simeq}S^p \vee S^q \vee S^{p+k-1}\vee S^{q+k-1}
    \] 
    such that \(\varepsilon\circ\varphi_\tau\simeq w_1\check{+}w_2\). Then the following diagram of homotopy cofibrations commutes up to homotopy
    \begin{equation*}
        \begin{tikzcd}[row sep=3em, column sep = 3em]
            S^{n+k-2} \arrow[d, equals] \arrow[r, "\varphi_\tau"] & S^p \vee S^q \vee S^{p+k-1}\vee S^{q+k-1} \arrow[d, "\varepsilon"] \arrow[r] & \mathcal{G}_\tau^k(S^p \times S^q) \arrow[d, dashed, "\simeq"]\\
            S^{n+k-2} \arrow[r, "w_1\check{+}w_2"] & S^p \vee S^q \vee S^{p+k-1}\vee S^{q+k-1} \arrow[r] & (S^p \times S^{q+k-1})\#(S^q \times S^{p+k-1})
        \end{tikzcd}
    \end{equation*}
    and the dashed arrow, i.e.~the induced map of homotopy cofibres, is a homology isomorphism and therefore a homotopy equivalence by Whitehead's Theorem, since both spaces are simply-connected. It is therefore sufficient to show that such an \(\varepsilon\) exists. 

    Take \(\widetilde{\tau}\) as in Lemma \ref{lem:taubar_whitehead}. Since a homotopy equivalence \(\varepsilon\) is in particular a map out of a wedge, it is defined by its precomposition with each inclusion. For our \(\varepsilon\) we set
    \[
        \varepsilon\circ i_p=i_p\text{\;, \;} \varepsilon\circ i_q=i_q\text{\;, \;} \varepsilon\circ i_{p+k-1}=i_{p+k-1}\text{\; and \;} \varepsilon\circ i_{q+k-1}=i_{q+k-1} - i_q\circ\Sigma\widetilde{\tau}.
    \]
    This \(\varepsilon\) is indeed a homotopy equivalence since we need only ensure it induces a homology isomorphism, which it evidently does since \(i_q\circ\Sigma\widetilde{\tau}\) vanishes in homology. Applying Proposition \ref{prop:varphi_tau'}, there is a sequence of homotopies
    \begin{align*}\label{eq:big_sequence_of_homotopies}
        \begin{split}
        \varepsilon\circ\varphi_\tau \simeq &~  \varepsilon\circ[i_p,i_{q+k-1}] + \varepsilon\circ[i_q,i_{p+k-1}] +  \varepsilon\circ[i_p,i_q\circ\Sigma\widetilde{\tau}] \\
        \simeq &~  [\varepsilon\circ i_p,\varepsilon\circ i_{q+k-1}] + [\varepsilon\circ i_q,\varepsilon\circ i_{p+k-1}] + [\varepsilon\circ i_p,\varepsilon\circ i_q\circ\Sigma\widetilde{\tau}] \\
        \simeq &~  [i_p,i_{q+k-1} - i_q\circ\Sigma\widetilde{\tau}] + [i_q,i_{p+k-1}] + [i_p,i_q\circ\Sigma\widetilde{\tau}] \\
        \simeq &~  [i_p,i_{q+k-1}] - [i_p, i_q\circ\Sigma\widetilde{\tau}] + [i_q,i_{p+k-1}] + [i_p,i_q\circ\Sigma\widetilde{\tau}] \\ 
        \simeq &~ [i_p,i_{q+k-1}] + [i_q,i_{p+k-1}].
        \end{split}
    \end{align*}
    By definition \([i_p,i_{q+k-1}] + [i_q,i_{p+k-1}]\simeq \nabla\circ([i_p,i_{q+k-1}] \check{+} [i_q,i_{p+k-1}])\), where \(\nabla\) denotes the fold map. Since the inclusions in these two Whitehead products are all distinct, we have the following homotopy commutative diagram (in which we omit the necessary permutation of the wedge summands in the second column, by abuse of notation)
    \begin{equation*}
        \begin{tikzcd}[row sep=3em, column sep = 10em]
            S^{n+k-2} \arrow[r, "{[i_p,i_{q+k-1}]\check{+}[i_q,i_{p+k-1}]}"] \arrow[d, equals] & (S^p \vee S^q \vee S^{p+k-1}\vee S^{q+k-1})\vee(S^p \vee S^q \vee S^{p+k-1}\vee S^{q+k-1}) \arrow[d, "\nabla"] \\
            S^{n+k-2} \arrow[r, "w_1\check{+}w_2"] & (S^p \vee S^{q+k-1}) \vee (S^q \vee S^{p+k-1}).
        \end{tikzcd}
    \end{equation*}
    Thus \(\varepsilon\circ\varphi_\tau\simeq [i_p,i_{q+k-1}] + [i_q,i_{p+k-1}] \simeq \nabla\circ([i_p,i_{q+k-1}] \check{+} [i_q,i_{p+k-1}])\simeq w_1\check{+}w_2\).
\end{proof}

\begin{warn}
    We should note that this does not indicate that one could prove for two Poincar\'e Duality complexes \(M\) and \(N\) that
    \(
        \mathcal{G}_\tau^k(M\times N)\simeq(M\times\mathcal{G}_\tau^k(N))\#(N\times\mathcal{G}_\tau^k(M)).
    \)
    This is known to fail if \(M\) and \(N\) are not spheres \cite{gitler-ldm}*{Lemma 2}. 
\end{warn}

\section{Applications}\label{sec:apps}

\subsection{Iterated Gyrations} 

Let \(M\) be a path-connected \(n\)-dimensional Poincar\'e Duality complex with a single \(n\)-cell, let \(r \geq 2\) be an integer and let \(K=\lbrace k_1, \dots ,k_r\rbrace\) be a set of integers \(k_i\geq 2\) for all \(i=1,\dots,r\). Taking the convention that \(k_0=1\), let \(n_i=n+\sum_{l=0}^{i-1}(k_l-1)\) and take homotopy classes \(\tau_i\in\pi_{k_i-1}(\mathrm{SO}(n_i))\) for each \(i=1,\dots,r\). Let \(T=\lbrace\tau_1,\dots\tau_r\rbrace\) be the set of these homotopy classes and call this the \textit{twisting set}. The \textit{iterated \(K\)-gyration of \(M\) by \(T\)} is given by 
\[
    \mathcal{G}^K_T(M):=\mathcal{G}_{\tau_r}^{k_r}(\dots(\mathcal{G}_{\tau_1}^{k_1}(M))\dots).
\]

From this definition one easily derives analogues to Fico's Lemmata from our work on ordinary gyrations.

\begin{cor}\label{cor:iterated_gyration_connsum}
Let \(M\) and \(N\) be two path-connected \(n\)-dimensional Poincar\'e Duality complexes, each with a single \(n\)-cell. If \(k_i \leq n_i-2\) for all \(r=1,\dots,r\) then \(\mathcal{G}^K_T(M \# N)\simeq\mathcal{G}^K_T(M)\#\mathcal{G}^K_T(N)\). \qed
\end{cor}

To state the result for iterated gyrations of sphere products in a succinct fashion, for a subset \(L\subseteq K\) let \(\Sigma_L=\sum_{k_i\in L}(k_i-1)\) and fix \(\Sigma_\emptyset=1\).

\begin{cor}\label{cor:iterated_gyration_sphereprod}
    If \(2\leq p\leq q\) such that \(n=p+q\) and \(k_i\leq n_i - p -1\) for each \(i=1,\dots,r\) then 
    \[
    \pushQED{\qed} 
        \mathcal{G}^K_T(S^p \times S^q)\simeq\underset{L\subseteq K}\#\left(\left(S^{p+\Sigma_L-1} \times S^{q+\Sigma_{K\backslash L}-1}\right)\#\left(S^{p+\Sigma_{K\backslash L}-1} \times S^{q+\Sigma_L-1}\right)\right).
    \qedhere
    \popQED
    \]
\end{cor}

\begin{rem}
Iterated gyrations are of particular prominence in relation to the Fan-Wang Conjecture, concerning certain polyhedral products \cite{fanwang}. This context overlaps significantly with toric topology, for which Corollary \ref{cor:iterated_gyration_sphereprod} may have further implications for generalisations of results of McGavran \cite{mcgavran_connsums}*{Theorem 3.5} (see also \cite{bosio_meersseman}*{Theorem 6.3}).
\end{rem}

\subsection{Gyration Stability}\label{subsec:gystab}

The notion of gyration stability has been explored for several cases (see for example \cites{ChenTher:gy_stab, HuangTher_stabilization}) and is formulated thusly.

\begin{defn}
    Let \(M\) be a path-connected \(n\)-dimensional Poincar\'e Duality complex with a single \(n\)-cell. For a given \(k\geq2\), \(M\) is called \textit{\(\mathcal{G}^k\)-stable} if \(\mathcal{G}^{k}_{\tau}(M)\simeq\mathcal{G}^{k}_{\omega}(M)\) for all twistings \(\tau,\omega\in\pi_{k-1}(\mathrm{SO}(n))\). 
\end{defn}
    
When the context is clear this property is called \textit{gyration stability}. Note that \(M\) is \(\mathcal{G}^k\)-stable if and only if for all twistings \(\tau\) there is a homotopy equivalence \(\mathcal{G}^{k}_{\tau}(M)\simeq\mathcal{G}^{k}_0(M)\). An immediate corollary to our generalisation of Fico's Lemmata is the following. 

\begin{cor}\label{cor:fico_stability}
Let \(M\) and \(N\) be two path-connected \(n\)-dimensional Poincar\'e Duality complexes, each with a single \(n\)-cell. Then
    \begin{enumerate}
        \item[(i)] if \(k \leq n-2\) and \(M\) and \(N\) are \(\mathcal{G}^k\)-stable, then so is \(M \# N\);
        \item[(ii)] if \(2\leq p\leq q\) such that \(n=p+q\) and \(k\leq q-1\), then the sphere product \(S^p \times S^q\) is \(\mathcal{G}^k\)-stable. \qed
    \end{enumerate}
\end{cor}

\begin{rem}
    Stability results for iterated gyrations are easily extracted for \(S^p\times S^q\). If the conditions of Corollary \ref{cor:iterated_gyration_sphereprod} hold, then clearly \(S^p\times S^q\) has the homotopy type of the given connected sum of binary sphere products, regardless of the twisting set. 
\end{rem}

\subsection{Inertness}

Let \(M\) be a simply-connected \(n\)-dimensional Poincar\'e Duality complex. In the context of the homotopy cofibration
\[
    S^{n-1}\xrightarrow{f_{M}}\overline{M}\xrightarrow{i_{M}} M
\]
the map \(f_M\) is \textit{inert} if \(i_M\) has a right homotopy inverse after looping, and we say that the complex \(M\) has \textit{inertness}. This integral version of a property from rational homotopy theory is due to Theriault \cite{t20}. Inertness is a very useful property if one knows the homotopy type of \(\overline{M}\), since it allows for computation of the homotopy type of the based loop space \(\Omega M\). The key result given below is due to Theriault \cite{t20}*{Theorem 9.1(c)} in the simply-connected case and later, for the path-connected case, by Huang \cite{huang_inertness24}*{Theorem 10.6}.

\begin{thm}[\cites{t20, huang_inertness24}]\label{thm:inert conn sums}
    Let \(M\) and \(N\) be path-connected \(n\)-dimensional Poincar\'e Duality complexes with \(n\geq 3\), both of which have a single \(n\)-cell. If \(M\) has inertness, then so does \(M\#N\). \qed
\end{thm}

A good example is when \(M\simeq S^p\times S^q\), for which the Hilton-Milnor Theorem implies that \(\Omega(S^p\times S^q)\) retracts off \(\Omega(S^p \vee S^q)\), implying that \(S^p\times S^q\) has inertness. Applying Theorem \ref{thm:inert conn sums} and Theorem \ref{thm:gy_sphereprods'} implies that \(\mathcal{G}_\tau^k(S^p \times S^q)\) has inertness whenever \(k\leq q-1\). 

For the other part of Fico's Lemmata, we deduce an inertness result for gyrations of connected sums if one of the summands is in a particular family of Poincar\'e Duality complexes: let \(m\) and \(n\) be integers such that \(3 \leq n\) and \(2\leq m \leq n-m\), and take \(\mathcal{M}\) be the family of \((m-1)\)-connected \(n\)-dimensional Poincar\'e Duality complexes \(M\) such that there exists a map \(S^m\rightarrow M\) with a left homotopy inverse.

\begin{cor}
    Let \(M\) and \(N\) be two \(n\)-dimensional Poincar\'e Duality complexes, \(N\) being path-connected with a single \(n\)-cell, and such that \(M\in\mathcal{M}\). Then for all indices \(k \leq n-2\) and all twistings \(\tau\in\pi_{k-1}(\mathrm{SO}(n))\) the gyration \(\mathcal{G}_\tau^k(M\#N)\) has inertness.
\end{cor}

\begin{proof}
    By Theorem \ref{thm:gy_connsums} we have \(\mathcal{G}_\tau^k(M\#N) \simeq \mathcal{G}_\tau^k(M) \# \mathcal{G}_\tau^k(N)\). Since \(M\in\mathcal{M}\), by \cite{StanTher_skeleton-coH}*{Theorem 6.8} the gyration \(\mathcal{G}_\tau^k(M)\) has inertness. Then apply Theorem \ref{thm:inert conn sums}.
\end{proof}

\bibliographystyle{amsplain}
\bibliography{bib}

\end{document}